\input ./style/arxiv-vmsta.cfg
\documentclass[numbers,compress,v1.0.1]{vmsta}
\usepackage{vtexbibtags}
\usepackage{authorquery}

\volume{5}
\issue{4}
\pubyear{2018}
\firstpage{415}
\lastpage{428}
\aid{VMSTA109}
\doi{10.15559/18-VMSTA109}
\articletype{research-article}

\startlocaldefs
\allowdisplaybreaks

\newtheorem{theorem}{Theorem}
\newtheorem{lemma}{Lemma}

\newtheorem{proposition}{Proposition}
\theoremstyle{definition}
\newtheorem{remark}{Remark}
\newtheorem{definition}{Definition}

\newcommand{\Lb} {{\mathbb L}}
\newcommand{\Nb} {{\mathbb N}}
\newcommand{\Rb} {{\mathbb R}}

\newcommand{\Cs} {{\mathcal C}}
\newcommand{\Fs} {{\mathcal F}}
\newcommand{\Ps} {{\mathcal P}}
\newcommand{\ep}{\varepsilon}
\renewcommand{\phi}{\varphi}
\newcommand{\con}{\triangleleft}
\newcommand{\pn}{(P_{\ell})_{\ell\geq1}}

\endlocaldefs

\begin{aqf}
\querytext{Q1}{Isn't "the scaling factor in front" just the "multiple" mentioned before? This query applies to other appearences of "scaling factor", too.}
\querytext{Q2}{Can the natural properties of a fractional Brownian motion be considered as drawbacks? They are as they are. The word 'limitations' could be better in this context.}
\querytext{Q3}{What is the relation of this 'alpha', tending to infinity, to  "the scaling factor in front of" (Brownian motion) converging to zero?}
\querytext{Q4}{Please check the edited sentence.}
\end{aqf}
\begin{document}

\begin{frontmatter}
\pretitle{Research Article}

\title{Asymptotic arbitrage in fractional mixed markets}

\author[a]{\inits{F.}\fnms{Fernando}~\snm{Cordero}\thanksref{cor1}\ead[label=e1]{fcordero@techfak.uni-bielefeld.de}} %
\thankstext[type=corresp,id=cor1]{Corresponding author.}
\author[b]{\inits{I.}\fnms{Irene}~\snm{Klein}\ead[label=e2]{irene.klein@univie.ac.at}}
\author[c]{\inits{L.}\fnms{Lavinia}~\snm{Perez-Ostafe}\ead[label=e3]{lavinia.ostafe@univie.ac.at}}
\address[a]{Faculty of Technology,
\institution{University of Bielefeld},
Universit\"{a}tsstr. 25, 33615 Bielefeld, \cny{Germany}}
\address[b]{Department of Statistics and Operations Research,
\institution{University of Vienna},
Oskar-Morgenstern-Platz 1, 1090 Vienna, \cny{Austria}}
\address[c]{Department of Mathematics,
\institution{University of Vienna}, Oskar-Morgenstern-Platz 1, 1090 Vienna,
\cny{Austria}}



\markboth{F. Cordero et al.}{Asymptotic arbitrage in fractional mixed markets}

\begin{abstract}
We consider a family of mixed processes given as the sum of a
fractional Brownian motion with Hurst parameter $H\in(3/4,1)$ and a
multiple of an independent standard Brownian motion, the family being
indexed by the scaling factor\querymark{Q1} in front of the Brownian motion. We
analyze the underlying markets with methods from large financial
markets. More precisely, we show the existence of a strong asymptotic
arbitrage (defined as in \xch{Kabanov and Kramkov [Finance Stoch. {2}({2}), {143}--{172} ({1998})]}{\cite{Kakra}}) when the scaling factor
converges to zero. We apply a result
of \xch{Kabanov and Kramkov [Finance Stoch. {2}({2}), {143}--{172} ({1998})]}{\cite{Kakra}} that
characterizes the notion of strong asymptotic arbitrage in terms of the entire asymptotic
separation of two sequences of probability measures. The main part of
the paper consists of proving the entire separation and is based on a
dichotomy result for sequences of Gaussian measures and the concept of
relative entropy.
\end{abstract}

\received{\sday{18} \smonth{5} \syear{2018}}
\revised{\sday{12} \smonth{7} \syear{2018}}
\accepted{\sday{12} \smonth{7} \syear{2018}}
\publishedonline{\sday{20} \smonth{8} \syear{2018}}
\end{frontmatter}
\setcounter{footnote}{0}

\section{Introduction}
Empirical studies of financial time series led to the conclusion that
the log-return increments exhibit long-range dependence. This fact
supports the idea of modelling the randomness of a risky asset using a
fractional Brownian motion with Hurst parameter $H>1/2$. However,
markets driven by a fractional Brownian motion have been extensively
disputed, as this motion fails to be a semimartingale and, hence, they
allow for a free lunch with vanishing risk (see \cite{G:R:S:2008}).

Many attempts were proposed to overcome this drawback\querymark{Q2} of the fractional
Brownian motion. In this work, we deal with the regularization method
proposed by Cheridito in \cite{Ch,Ch1} when $H>3/4$. This method
consists in adding to the fractional Brownian motion a multiple of an
independent Brownian motion, the resulting process, called \textit
{mixed fractional Brownian motion}, being Gaussian with the long-range
dependence property. Moreover, as shown in \cite{Ch,Ch1}, when $H>3/4$
the mixed fractional Brownian motion is equivalent to a multiple of a
Brownian motion. Therefore, a Black--Scholes type model in which the
randomness of the risky asset is driven by a mixed fractional Brownian
motion is arbitrage free and complete. We call such a model a \textit
{mixed fractional Black--Scholes model}.

On one hand the fractional Black--Scholes model admits arbitrage. On the
other hand, when the Hurst parameter $H>3/4$, adding a Brownian
component (in the above explained way) makes the arbitrage disappear.
In this paper we aim to go a step further and study the sensitivity to
arbitrage of the mixed fractional Black--Scholes model when the Brownian
component asymptotically vanishes. In \cite{CP,CP2} it was argued
that a good way of seeing the sensitivity to arbitrage of a market when
one of its parameters converges to zero (or infinity), is to consider
the family of markets indexed by the corresponding parameter and to use
methods from large financial markets. To be precise, we study the
asymptotic arbitrage opportunities in the sequence of mixed fractional
Black--Scholes models when the scaling factor in front of the Brownian
motion converges to zero. We focus on the notion of \textit{strong
asymptotic arbitrage} (SAA) introduced by Kabanov and Kramkov in \cite
{Kakra} as the possibility of getting arbitrarily rich with probability
arbitrarily close to one by taking a vanishing risk. Our model fits the
standard framework of large financial markets, as each mixed fractional
Black--Scholes model is arbitrage free (and even complete). We point out
that the existence of arbitrage in the limiting market does not
directly imply the existence of any kind of asymptotic arbitrage in the
approximating sequence of mixed markets. In \cite{Kakra} the existence
of strong asymptotic arbitrage was shown to be equivalent to the entire
asymptotic separation of the sequence of objective probability measures
and the sequence of equivalent martingale measures.

In order to show the existence of strong asymptotic arbitrage in the
sequence of mixed fractional Black--Scholes models we use the result of
\cite{Kakra} that was mentioned above. That means we show that the
sequence of objective probability measures is entirely asymptotically
separable from the sequence of equivalent martingale measures. Our main
contribution is the proof of this entire asymptotic separability in the
given model. We use the notion of relative entropy and a dichotomy
result for sequences of Gaussian measures. Indeed, inspired by the work
of Cheridito \cite{Ch,Ch1}, we first show, for each fixed market, that
the entropy of the objective probability measure relative to the
equivalent martingale measure, both restricted to a discrete partition,
converges to infinity. Our proof then follows using tightness arguments
for the sequence of Radon--Nikodym derivatives of the objective
probability measures with respect to equivalent martingale measures and
the fact that two sequences of Gaussian measures are either mutually
contiguous or entirely separable. The latter is known in the literature
as the equivalence/singularity dichotomy for sequences of Gaussian
processes, see \cite{Eag81}.

The paper is structured as follows. In Section~\ref{S1}, we set the
mixed fractional Black--Scholes model and recall the framework of the
large financial market. At the end of this part, we state the main
result (Theorem~\ref{t1}). Section \ref{S2} is dedicated to the proof
of Theorem~\ref{t1}, whereas Section \ref{S3} provides a discussion
about the existence of strong asymptotic arbitrage using self-financing
strategies constrained to jump only in a finite set of times. We end
our work with Appendix~\ref{A1} in which we recall the definition of
relative entropy and an equivalent characterization in terms of the
Radon--Nikodym derivative.
\section{Preliminaries and main results}\label{S1}

\subsection{Setting the model}\label{S1.1}
Let $(\varOmega, \Fs, P)$ be a probability space.

\begin{definition}
A fractional Brownian motion $Z^H=(Z^H_t)_{t\geq0}$ with Hurst
parameter $H\in(0,1)$ is a continuous centred Gaussian process with
covariance function
\[
\mathrm{Cov}\bigl(Z^H_t,Z^H_s
\bigr)=E\bigl(Z^H_tZ^H_s\bigr)=
\frac{1}2 \bigl(t^{2H}+s^{2H}-|t-s|^{2H}
\bigr),\quad  s,t\geq0.
\]
In particular, $Z^{\frac{1}2}$ is a standard Brownian motion.
\end{definition}
A linear combination of different fractional Brownian motions is
refered in the literature as a \textit{mixed fractional Brownian
motion}. In order to avoid localization arguments we only consider
finite time horizon processes. In addition we focus on linear
combinations of a standard Brownian motion $(B_t)_{t\in[0,1]}$ and an
independent fractional Brownian motion $(Z_t^H)_{t\in[0,1]}$ with
Hurst parameter $H\in(3/4,1)$, both defined on $(\varOmega, \Fs, P)$.
Cheridito shows in \cite{Ch,Ch1} that, for each $\alpha\in\Rb$ the
mixed process $M^{H,\alpha}:=(M_t^{H,\alpha})_{t\in[0,1]}$ defined
by
\[
M_t^{H,\alpha}:=\alpha\, Z_t^H+
B_t,\quad t\in[0,1],
\]
is equivalent to a Brownian motion. By this we mean that the measure
$Q^{H,\alpha}$ induced on $\Cs[0,1]$ by $M^{H,\alpha}$ and the
Wiener measure $Q_W$ (induced by the Brownian motion on $\Cs[0,1]$)
are equivalent. As a consequence, the process $M^{H,\alpha}$ is a
$(\Fs_t^{H,\alpha})_{t\in[0,1]}$-semimartingale, where, for each
$t\geq0$, $\Fs_t^{H,\alpha}:=\overline{\sigma((M_s^{H,\alpha
})_{s\in[0,t]})}$ is the right-continuous natural filtration augmented
by the nullsets.

Now, for each $\alpha>0$, we call by the \textit{$\alpha$-mixed
fractional Black--Scholes model} the financial market consisting of a
risk free asset normalized to one and a risky asset $(S_t^{H,\alpha
})_{t\in[0,1]}$ given by
%
\begin{equation}
\label{smm} S_t^{H,\alpha}:=S_0^{H,\alpha} \exp
\biggl( \biggl(\mu-\frac{\sigma
^2}{2\alpha^2} \biggr)\, t+\sigma \biggl(Z_t^H+
\frac{1}{\alpha} B_t \biggr) \biggr),\quad t\in[0,1],
\end{equation}
where $\mu\in\Rb$ and $\sigma>0$ represent the drift and the
volatility of the asset.\footnote{$S^{H,\alpha}$ is the solution of
$dS_t^{H,\alpha}=\mu\,S_t^{H,\alpha}\,dt+\sigma\, S_t^{H,\alpha}\,
d(Z^H+\frac{1}{\alpha} B)_t$.} We denote by $X:=(X_{t})_{t\in[0,1]}$
the coordinate process in $\Cs[0,1]$ and we define the process
$S^\alpha:=(S^\alpha_t)_{t\in[0,1]}$ as
%
\begin{equation}
\label{sm} S^\alpha_t:=S^\alpha_0
\exp \biggl( \biggl(\mu-\frac{\sigma
^2}{2\alpha^2} \biggr)\, t+\frac{\sigma}{\alpha}
X_t \biggr),\quad t\in[0,1].
\end{equation}
From the above discussion, we conclude that $S^{H,\alpha}$ under $P$
is equivalent to $S^\alpha$ under $Q_W$, which is a martingale when
$\mu=0$. For a general drift, we denote by $Q_{\frac{\mu\alpha
}{\sigma}}$ the measure induced on $\Cs[0,1]$ by the Brownian motion
with drift $-\frac{\mu\alpha}{\sigma}$ (in particular $Q_0=Q_W$).
Thanks to the Girsanov theorem, the process $S^{H,\alpha}$ under $P$
is also equivalent to $S^\alpha$ under $Q_{\frac{\mu\alpha}{\sigma
}}$, which is a martingale. Therefore, the $\alpha$-mixed fractional
Black--Scholes model with the filtration $(\Fs_t^{H,\alpha})_{t\in
[0,1]}$ has a unique equivalent martingale measure, and therefore is
arbitrage-free and complete.
\subsection{Asymptotic arbitrage}
In this work, we treat the collection of $\alpha$-mixed fractional
Black--Scholes models with methods from large financial markets. This
idea is formalized in the following definition.
%
\begin{definition}[The large mixed fractional market]
We call by \textit{large mixed fractional market} the family of
$\alpha$-mixed fractional Black--Scholes models, $\alpha>0$, i.e. the
family of markets
\[
\Lb^H:= \bigl(\varOmega,\Fs,\bigl(\Fs_t^{H,\alpha}
\bigr)_{t\in
[0,1]},P,S^{H,\alpha} \bigr)_{\alpha>0}.
\]
\end{definition}
We aim to study the presence of asymptotic arbitrage in the large
financial mixed fractional market when $\alpha$ tends to infinity,\querymark{Q3}
i.e. when the Brownian component asymptotically disappears. More
precisely, we intend to investigate, using methods of \cite{Kakra},
the presence of a so-called \textit{strong asymptotic arbitrage}. The
latter is an analogue concept of arbitrage but for sequences of markets
rather than for a single market model. Intuitively, this kind of
arbitrage for sequences of markets gives the possibility of getting
arbitrarily rich with probability arbitrarily close to one while taking
a vanishing risk. In order to make this idea precise we first specify
the set of admissible trading strategies.
%
\begin{definition}[Admissible trading strategy]\label{SFAS}
A \textit{trading strategy} for $S^{H,\alpha}$ is a real-valued
predictable $S^{H,\alpha}$-integrable stochastic process $\varPhi:=(\varPhi
_t)_{t\in[0,1]}$. The trading strategy is said to be \textit
{admissible} if there is $m\in\Rb_+$ such that for all $t\in[0,1]$:
$(\varPhi\cdot S^{H,\alpha})_t\geq-m$ almost surely.
\end{definition}
Now we proceed to recall the definition of strong asymptotic arbitrage of \cite{Kakra}.
%
\begin{definition}\label{SAA}
A strong asymptotic arbitrage (SAA) is said to exist in the large mixed
fractional market as $\alpha$ tends to infinity if there exists a
sequence $(\alpha_\ell)_{\ell\geq1}$ converging to infinity and a
sequence $(\varPhi_\ell)_{\ell\geq1}$, where $\varPhi_\ell$ is an
admissible trading strategy for $S^{H,\alpha_\ell}$, such that
\begin{enumerate}
\item$(\varPhi_\ell\cdot S^{H,\alpha_\ell})_t\geq-m_\ell,\quad$
$t\in[0,1]$, $\ell\geq1$,
\item$\lim_{\ell\to\infty}P ((\varPhi_\ell\cdot S^{H,\alpha
_\ell} )_1\geq M_\ell)=1$,
\end{enumerate}
where $m_k$ and $M_k$ are sequences of positive real numbers converging
to zero and to infinity, respectively.
\end{definition}

\begin{remark}
This definition is equivalent to the notion of strong asymptotic
arbitrage of the first kind as given in \cite{Kakra}. This is
trivially seen by taking $V^\ell_0(\varPhi_\ell)=\frac{m_\ell
}{M_\ell}$ and $V^\ell_t(\varPhi_\ell)=\frac{m_\ell}{M_\ell}+\frac{1}{M_\ell}(\varPhi_\ell\cdot S^{H,\alpha_\ell})_t$ that
the SAA1 of \cite{Kakra} can be obtained from our SAA. It is equally trivial to get a SAA from the SAA1.
The notion is further equivalent to the strong asymptotic arbitrage of
the second kind from \cite{Kakra} as it is shown there that SAA1 and
SAA2 are equivalent and hence can be subsumed under the name SAA.
\end{remark}

Our approach to show the existence of arbitrage of this kind  will be not constructive. Instead,
we use an equivalent characterization of strong asymptotic arbitrage
based on the notion of \textit{entire asymptotic separability of
sequences of measures}, which is defined as follows.
%
\begin{definition}\label{asymsep}
The sequences of probability measures $\pn$ and $(Q_\ell)_{\ell\geq
1}$ are said to be
entirely asymptotically separable if
there exists a subsequence $\ell_k$ and a sequence of sets $A_{k}\in
\mathcal{F}^{\ell_k}$ such that $\lim_{k\to\infty}P_{\ell_k}(A_{k})=1$
and $\lim_{k\to\infty}Q_{\ell_k}(A_k)=0$. In this case we write
$\pn\vartriangle(Q_\ell)_{\ell\geq1}$. In addition, two families
of probability
measures $(P^\alpha)_{\alpha>0}$ and $(Q^\alpha)_{\alpha>0}$ are
said to be entirely asymptotically
separable, and we write $(P^\alpha)_{\alpha>0}\vartriangle(Q^\alpha
)_{\alpha>0}$, if there is a sequence
$(\alpha_\ell)_{\ell\geq1}$ converging to infinity such that
$(P^{\alpha_\ell})_{\ell\geq1}\vartriangle(Q^{\alpha_\ell
})_{\ell\geq1}$.
\end{definition}

The precise relation between this notion and the existence of SAA is
given in \cite[Proposition 4]{Kakra}. In the case of complete markets,
this result takes the following simple form.
%
\begin{proposition}\label{pkk}
Consider a large financial market $(\varOmega^\alpha,\Fs^\alpha,(\Fs
_t^\alpha)_{t\in[0,T]}, P^\alpha)_{\alpha>0 }$ and assume that each
small market is complete. For each $\alpha> 0$, let $Q^\alpha\sim
P^\alpha$ be the unique equivalent martingale measure. Then the
following conditions are equivalent
\begin{enumerate}
\item There is a SAA.
\item$(P^\alpha)_{\alpha>0}\vartriangle(Q^\alpha)_{\alpha>0}$.
\end{enumerate}
\end{proposition}
Therefore the study of SAA in $\Lb^H$ reduces to determining whether
$(Q^{H,\alpha})_{\alpha>0}$ is entirely asymptotically separable from
$(Q_{\frac{\mu\alpha}{\sigma}})_{\alpha>0}$ or not.\footnote{In the case when $\mu=0$, the study of SAA reduces to showing that
$(Q^{H,\alpha})_{\alpha>0}\vartriangle Q_W$.}
\subsection{Main result}\label{S1.3}
We state now our main result.
%
\begin{theorem}\label{t1}
There exists a strong asymptotic arbitrage in the large mixed
fractional market $\Lb^H$ for $\alpha\to\infty$.
\end{theorem}
As mentioned we will show that $(Q^{H,\alpha})_{\alpha>0}\vartriangle
(Q_{\frac{\mu\alpha}{\sigma}})_{\alpha>0}$.
\section{Proof of Theorem~\ref{t1}}\label{S2}
In order to prove Theorem~\ref{t1}, we provide a series of lemmas from
which the desired result is obtained as a direct consequence. Before
proceeding, we introduce 
some notations.

Following the lines of \cite{Ch,Ch1}, we define, for all $n\in\Nb$,
$Y_n:C[0,1]\rightarrow\Rb^n$ by:
\[
Y_n(\omega)= \biggl(\omega \biggl(\frac{1}{n} \biggr)-
\omega(0),\omega \biggl(\frac{2}{n} \biggr)-\omega \biggl(
\frac{1}{n} \biggr),\ldots ,\omega(1)-\omega \biggl(\frac{n-1}n
\biggr) \biggr)^T
\]
and denote by $Q^{H,\alpha,n}$ and $Q_{\frac{\mu\alpha}{\sigma}}^n$
the restrictions of $Q^{H,\alpha}$ and $Q_{\frac{\mu\alpha}{\sigma
}}$ to the $\sigma$-algebra $\Fs_n:=\sigma(Y_n)$.
We fix the Hurst parameter $H\in(3/4,1)$ and we avoid to mention the
dependence on it by setting $Q^{H,\alpha}\equiv Q^{\alpha}$ and
$Q^{H,\alpha,n}\equiv Q^{\alpha,n}$.

We denote by $C_n$ the covariance matrix of the increments of the
fractional Brownian motion $Z^H$:
\[
C_n(i,j):=\textrm{Cov} \bigl(Z_{\frac{i}{n}}^H-Z_{\frac
{i-1}{n}}^H,Z_{\frac{j}{n}}^H-Z_{\frac{j-1}{n}}^H
\bigr),\quad1\leq i,j\leq n,
\]
and by $\lambda_1^n,\dots,\lambda_n^n$ its eigenvalues. Since the
matrix $C_n$ is symmetric and positive semi-definite, all the $\lambda
_i^n$, $1\leq i\leq n$, are real and nonnegative.

We moreover set
\[
\varSigma_0:=\frac{1}{n}I_n+\alpha^2
C_n\quad\textrm{and}\quad\varSigma _1:=
\frac{1}{n}I_n+\frac{1}{n^2} \frac{\mu^2\alpha^2}{\sigma^2}\,
1_{n\times n},
\]
where $I_n$ is the identity matrix and $1_{n\times n}$ is the $n\times
n$ matrix with all coefficients equal to $1$. Clearly, the matrices
$\varSigma_0$ and $\varSigma_1$ are positive definite and therefore invertible.

The proof of Theorem \ref{t1} strongly relies on the concept of
relative entropy (also called sometimes Kullback--Leibler divergence) of
the probability measure $Q^{\alpha,n}$ (respectively, $Q^{\alpha}$)
relative to $Q_{\frac{\mu\alpha}{\sigma}}^n$ (respectively,
$Q_{\frac{\mu\alpha}{\sigma}}$), denoted by $H(Q^{\alpha
,n}|Q_{\frac{\mu\alpha}{\sigma}}^n)$ (respectively, $H(Q^{\alpha
}|Q_{\frac{\mu\alpha}{\sigma}})$), see \cite[Section 6]{Hihi}. We
recall the definition of relative entropy and some relevant results in
the Appendix~\ref{A1}.

\begin{lemma}\label{relent}
For each $n\geq1$, we have
%
\begin{equation}
\label{eg} H \bigl(Q^{\alpha,n}|Q_{\frac{\mu\alpha}{\sigma}}^n \bigr)=
\frac
{1}{2} \biggl[\textrm{tr}\bigl(\varSigma_1^{-1}
\varSigma_0\bigr)-n+\frac{\mu
^2\alpha^2}{\sigma^2 n^2} 1_n^T
\varSigma_1^{-1}1_n+\ln \biggl(\frac
{\textrm{det} (\varSigma_1 )}{\textrm{det} (\varSigma
_0 )}
\biggr) \biggr],
\end{equation}
where $1_n\in\Rb^n$ is the vector with all coordinates equal to $1$,
and, for each square matrix $A$, $\textrm{tr}(A)$ and $\textrm
{det}(A)$ denote the trace and the determinant of $A$, respectively.
\end{lemma}
\begin{proof}
Note first that
\[
E_{Q^{\alpha,n}} \bigl[Y_nY_n^T \bigr]=
\varSigma_0 \quad\textrm {and}\quad E_{Q_{\frac{\mu\alpha}{\sigma}}^n}
\bigl[Y_nY_n^T \bigr]=\varSigma_1.
\]
Note also that
\[
E_{Q^{\alpha,n}} [Y_n ]=0_n\quad\textrm{and}\quad
E_{Q_{\frac{\mu\alpha}{\sigma}}^n} [Y_n ]=-\frac{\mu
\alpha}{\sigma n} 1_n,
\]
where $0_n\in\Rb^n$ is the vector with all coordinates equal to $0$.
Since $Y_n$ is a Gaussian vector under the two measures, the result
follows using Lemma \ref{a1.2} and performing a straightforward calculation.
\end{proof}
Using standard properties of the trace and the determinant, it is not
difficult to see that
%
\begin{equation}
\label{trdet} \textrm{tr} (\varSigma_0 )=\sum\limits
_{i=1}^n
\biggl(\frac
{1}{n}+\alpha^2\lambda_i^n
\biggr)\quad\textrm{and}\quad\ln \bigl(\textrm{det} (\varSigma_0 )
\bigr)=\sum\limits
_{i=1}^n\ln \biggl(\frac{1}{n}+
\alpha^2\lambda_i^n \biggr).
\end{equation}
We set $a_n:=\frac{1}{n}\frac{\mu^2\alpha^2}{\sigma^2}$ and note
that $\varSigma_1=\frac{1}{n}(I_n+a_n1_{n\times n})$.
The next lemma summarizes the properties of the matrix $\varSigma_1$.
%
\begin{lemma}\label{linalg}
For each $n>1$, the eigenvalues of $\varSigma_1$ are $1/n$ with
multiplicity $n-1$ and $\frac{1}{n}+a_n$ with multiplicity $1$. In
particular, we have
\[
\textrm{det} (\varSigma_1 )=\frac{n a_n+1}{n^n}.
\]
The inverse of $\varSigma_1$ is given by
\[
\varSigma_1^{-1}=n \biggl(I_n-
\frac{a_n}{n a_n+1}1_{n\times n} \biggr).
\]
\end{lemma}
\begin{proof}
Denote 
$d_n^\lambda:=\text{det}(\varSigma_1-\lambda I_n)=\text
{det} ( (\frac{1}{n}-\lambda )I_n+\frac
{a_n}{n}1_{n\times n} )$. Subtracting the row $i$ from the row
$i+1$, for each $1\leq i<n$, in the matrix $\varSigma_1-\lambda I_n$, we
see that $d_n^\lambda$ is equal to the determinant of the matrix
%
\begin{equation}
\begin{pmatrix}
\frac{1}{n}-\lambda+\frac{a_n}{n}&\frac{a_n}{n}&\frac{a_n}{n}&\cdots
&\frac{a_n}{n}\\
\lambda-\frac{1}{n}&\frac{1}{n}-\lambda&0&\ldots&0\\
0&\ddots&\ddots&\ddots&0\\
\vdots&\ddots&\lambda-\frac{1}{n}&\frac{1}{n}-\lambda&0\\
0&\cdots&0&\lambda-\frac{1}{n}&\frac{1}{n}-\lambda
\end{pmatrix} %
.
\end{equation}
%
Expanding the determinant by minors with respect to the last column we get
\[
d_n^\lambda= \biggl(\frac{1}{n}-\lambda
\biggr)d_{n-1}^{\lambda}+\frac
{a_n}{n} \biggl(
\frac{1}{n}-\lambda \biggr)^{n-1},\quad n>2.
\]
Iterating this identity, we obtain
\[
d_n^\lambda= \biggl(\frac{1}{n}-\lambda
\biggr)^{n-2}d_2^{\lambda
}+(n-2)\frac{a_n}{n}
\biggl(\frac{1}{n}-\lambda \biggr)^{n-1}= \biggl(
\frac{1}{n}-\lambda \biggr)^{n-1} \biggl(\frac{1}{n}-
\lambda+a_n \biggr).
\]
The first two statements follow. For the last statement, one can easily
check that
\[
\varSigma_1\times n \biggl(I_n-\frac{a_n}{n a_n+1}1_{n\times n}
\biggr)=I_n.
\]
This shows the desired result.
\end{proof}

\begin{lemma}\label{l1}
For all $n> 1$, we have
\[
\lim\limits
_{\alpha\rightarrow\infty}H \bigl(Q^{\alpha
,n}|Q_{\frac{\mu\alpha}{\sigma}}^n
\bigr)=\infty.
\]
%
\end{lemma}
\begin{proof}
Our starting point is Lemma~\ref{relent}. Evaluating each term
entering \eqref{eg}, we first obtain
%
\begin{align}
\label{term1.1} \text{tr}\bigl(\varSigma_1^{-1}
\varSigma_0\bigr)&=n\,\text{tr}(\varSigma_0)-
\frac{n
a_n}{n a_n+1}\,\text{tr}(1_{n\times n}\varSigma_0)
\nonumber
\\
&=\sum_{i=1}^n\bigl(1+n
\alpha^2\lambda_i^n\bigr)-\frac{n a_n}{n
a_n+1}
\bigl(1+\alpha^2\text{tr}(1_{n\times n}C_n)\bigr).
\end{align}
Note that
\[
\text{tr}(1_{n\times n}C_n)=\sum_{i,j=1}^n
C_n(i,j)=E \bigl[ \bigl(Z_1^H
\bigr)^2 \bigr]=1.
\]
Thus, 
taking $a_n=\frac{1}{n}\frac{\mu^2\alpha^2}{\sigma^2}$,
equation \eqref{term1.1} becomes
%
\begin{align}
\label{term1.2} \text{tr}\bigl(\varSigma_1^{-1}
\varSigma_0\bigr)&=n+\sum_{i=1}^nn
\alpha^2\lambda _i^n-\frac{\mu^2\alpha^2}{\mu^2\alpha^2+\sigma^2}
\bigl(1+\alpha^2\bigr).
\end{align}
For the third term in \eqref{eg}, using that $1_n^T1_{n\times
n}1_n=n^2$, one can easily derive that
%
\begin{align}
\label{term2} 1_n^T\varSigma_1^{-1}1_n=
\frac{n^2}{n a_n+1}=\frac{n^2\sigma^2}{\mu
^2\alpha^2+\sigma^2}.
\end{align}
For the last term in \eqref{eg}, we use \eqref{trdet} and Lemma~\ref
{linalg} to obtain
%
\begin{align}
\label{term3} \ln \biggl(\frac{\textrm{det} (\varSigma_1 )}{\textrm
{det} (\varSigma_0 )} \biggr)&=\ln(n a_n+1)-
\sum_{i=1}^n\ln \bigl(1+n
\alpha^2\lambda_i^n\bigr)
\nonumber
\\
&=\ln \biggl(\frac{\mu^2\alpha^2+\sigma^2}{\sigma^2} \biggr)-\sum_{i=1}^n
\ln\bigl(1+n\alpha^2\lambda_i^n\bigr).
\end{align}
Inserting \eqref{term1.2}, \eqref{term2} and \eqref{term3} in \eqref{eg} yields 
%
\begin{align}
\label{2entrop} H \bigl(Q^{\alpha,n}|Q_{\frac{\mu\alpha}{\sigma}}^n \bigr)&=
\frac{1}2 \Biggl[\sum\limits
_{i=1}^n \bigl(n
\alpha^2\lambda_i^n-\ln\bigl(1+n\alpha
^2\lambda_i^n\bigr)\bigr)-
\frac{\mu^2\alpha^4}{\mu^2\alpha^2+\sigma^2}
\nonumber\\
&\quad +\ln \biggl(\frac{\mu^2\alpha^2+\sigma^2}{\sigma^2} \biggr) \Biggr].
\end{align}
Since the trace is similarity-invariant, we deduce that
\[
\sum_{i=1}^n\lambda_i^n=
\textrm{tr}(C_n)=\sum_{i=1}^nC_n(i,i)=
\frac{1}{n^{2H-1}}.
\]
In addition, we have $\ln (\frac{\mu^2\alpha^2+\sigma
^2}{\sigma^2} )\geq0$. Therefore, \eqref{2entrop} leads to
%
\begin{align}
\label{lowb1} H \bigl(Q^{\alpha,n}|Q_{\frac{\mu\alpha}{\sigma}}^n \bigr)&
\geq \frac{\alpha^2}2 \biggl(n^{2-2H}-\frac{\mu^2\alpha^2}{\mu^2\alpha
^2+\sigma^2} \biggr)-
\frac{n}2 \ln\bigl(1+n\alpha^2\lambda_{\text
{max}}^n
\bigr)
\nonumber
\\
&\geq\frac{\alpha^2}2\bigl(n^{2-2H}-1\bigr)-\frac{n}2 \ln
\bigl(1+n\alpha ^2\lambda_{\text{max}}^n\bigr)
\nonumber
\\
&=\frac{1}2\ln \biggl(\frac{e^{\theta_n\alpha^2}}{(1+n\alpha
^2\lambda_{\text{max}}^n)^n} \biggr),
\end{align}
where $\theta_n:=n^{2-2H}-1>0$ and $\lambda_{\text{max}}^n=\max_{i=1\ldots n}\lambda_i^n$. The result follows taking the limit when
$\alpha$ tends to infinity in the previous expression.
\end{proof}
%
\begin{remark}\label{rmk0}
If $\mu=0$, using Lemma \ref{relent}, the previous result extends
directly to the case $n=1$.
\end{remark}

\begin{remark}\label{rmk1}
The above proof also gives us the relation between the relative entropy
of $Q^{\alpha,n}$ relative to $Q_{\frac{\mu\alpha}{\sigma}}^n$,
i.e.~$H (Q^{\alpha,n}|Q_{\frac{\mu\alpha}{\sigma}}^n
)$, and the relative entropy of $Q^{\alpha,n}$ relative to $Q_W^n$,
i.e.~$H (Q^{\alpha,n}|Q_W^n )$. Indeed, using \cite[Lemma
5.3]{Ch1} one can deduce from \eqref{2entrop} that
%
\begin{equation}
\label{rmkeq} H \bigl(Q^{\alpha,n}|Q_{\frac{\mu\alpha}{\sigma}}^n \bigr)=H
\bigl(Q^{\alpha,n}|Q_W^n \bigr)-\frac{1}2
\frac{\mu^2\alpha^4}{\mu
^2\alpha^2+\sigma^2}+\frac{1}2\ln \biggl(\frac{\mu^2\alpha^2+\sigma
^2}{\sigma^2} \biggr).
\end{equation}
\end{remark}

\begin{remark}
We point out that we also have
\[
\lim\limits
_{\alpha\rightarrow\infty}H \bigl(Q^{\alpha}|Q_{\frac
{\mu\alpha}{\sigma}} \bigr)=\infty.
\]
Indeed, we know from \cite[Lemma 5.3]{Ch1} that $\sup_n H
(Q^{\alpha,n}|Q_W^n )<\infty$, which directly implies that also
$\sup_n H (Q^{\alpha,n}|Q_{\frac{\mu\alpha}{\sigma}}^n
)<\infty$. Therefore, applying \cite[Lemma 6.3]{Hihi} we obtain
\[
H \bigl(Q^{\alpha}|Q_W \bigr)=\sup_nH
\bigl(Q^{\alpha
,n}|Q_W^n \bigr)\ \text{and}\ H
\bigl(Q^{\alpha}|Q_{\frac{\mu
\alpha}{\sigma}} \bigr)=\sup_nH
\bigl(Q^{\alpha,n}|Q_{\frac{\mu
\alpha}{\sigma}}^n \bigr).
\]
The statement then follows from the result for the restrictions.
\end{remark}

For each $n\geq1$, we denote the Radon--Nikodym derivative of
$Q^{\alpha,n}$ relative to $Q_{\frac{\mu\alpha}{\sigma}}^n$ by
\[
L_\alpha^n:=\frac{dQ^{\alpha,n}}{dQ_{\frac{\mu\alpha}{\sigma}}^n}.
\]
Using \cite[Lemma 6.1]{Hihi} (see Lemma~\ref{a1.2} in Appendix~\ref{A1}), we see that
\[
H \bigl(Q^{\alpha,n}|Q_{\frac{\mu\alpha}{\sigma}}^n \bigr)=E_{Q^{\alpha,n}}
\bigl[\ln\bigl(L_\alpha^n\bigr)\bigr]=E_{Q_{\frac{\mu\alpha}{\sigma
}}^n}
\bigl[L_\alpha^n\ln\bigl(L_\alpha^n\bigr)
\bigr].
\]

Moreover, let us recall the notion of $(Q^{\alpha,n})_{\alpha
>0}$-tightness: $(L_\alpha^n)_{\alpha>0}$ is $(Q^{\alpha,n})_{\alpha
>0}$-tight if the following holds:
\[
\lim_{N\rightarrow\infty}\limsup_{\alpha\to\infty}Q^{\alpha
,n}
\bigl(L_\alpha^n>N\bigr)=0.
\]

\begin{lemma}\label{l2}
For each $n>1$, the family $(L_\alpha^n)_{\alpha>0}$ is not
$(Q^{\alpha,n})_{\alpha>0}$-tight.
\end{lemma}
\begin{proof}
We know, by Lemma~\ref{l1}, that $E_{Q^{\alpha,n}}[\ln(L_\alpha
^n)]=H (Q^{\alpha,n}|Q_{\frac{\mu\alpha}{\sigma}}^n )$
tends to infinity when $\alpha$ tends to $\infty$. Since the measures
$Q^{\alpha,n}$ and $Q_{\frac{\mu\alpha}{\sigma}}^n$ are Gaussian,
the result follows as a direct application of the remark on \cite[p.
457]{Eag81} which says that tightness is equivalent to the boundedness
of the following two families: $E_{Q^{\alpha,n}}[\ln(L_\alpha^n)]$,
$\alpha>0$, and $\text{Var}_{Q^{\alpha,n}}[\ln(L_\alpha^n)]$,
$\alpha>0$.
\end{proof}

Before we can state and prove the last lemma of this section, we recall
now the definition of contiguity of sequences/families of probability measures.
%
\begin{definition}\label{con}
A sequence of probability measures $\pn$ is contiguous with respect to
the sequence of probability measures $(Q_\ell)_{\ell\geq1}$,
$\pn\con(Q_\ell)_{\ell\geq1}$, if for any sequence $A_\ell\in
\mathcal{F}^\ell$:
$\lim\limits_{\ell\rightarrow\infty}Q_\ell(A_\ell)=0\Rightarrow
\lim\limits_{\ell\rightarrow\infty}P_\ell(A_\ell)=0$. We say
that $\pn$ and $(Q_\ell)_{\ell\geq1}$ are mutually contiguous if
$\pn\con(Q_\ell)_{\ell\geq1}$ and
$(Q_\ell)_{\ell\geq1}\con\pn$, in which case we write $\pn
\triangleleft
\triangleright(Q_\ell)_{\ell\geq1}$.

These notions extend to families of probability measures $(P^\alpha
)_{\alpha>0}$ and
$(Q^\alpha)_{\alpha>0}$ as follows. We say that $(P^\alpha)_{\alpha
>0}$ is contiguous (resp. mutually
contiguous) to $(Q^\alpha)_{\alpha>0}$ if for every sequence $(\alpha
_\ell)_{\ell\geq
1}$ converging to infinity we have $(P^{\alpha_\ell})_{\ell\geq
1}\con(Q^{\alpha_\ell})_{\ell\geq1}$ (resp. $(P^{\alpha_\ell
})_{\ell\geq1}\triangleleft\triangleright(Q^{\alpha_\ell})_{\ell
\geq1}$), in which case we write $(P^\alpha)_{\alpha>0}\con
(Q^\alpha)_{\alpha>0}$ (resp. $(P^\alpha)_{\alpha>0}
\triangleleft\triangleright(Q^\alpha)_{\alpha>0}$).
\end{definition}

\begin{lemma}\label{finiten}
For each $n>1$, we have
\[
\bigl(Q^{\alpha,n}\bigr)_{\alpha>0}\vartriangle \bigl(Q_{\frac{\mu\alpha
}{\sigma}}^n
\bigr)_{\alpha>0}.
\]
\end{lemma}
\begin{proof}
Since, by Lemma~\ref{l2}, $(L_\alpha^n)_{\alpha>0}$ is not tight
with respect to $(Q^{\alpha,n})_{\alpha>0}$ we apply \cite[Lemma
V.1.6]{Jashi} and deduce that, for each $n> 1$, $(Q^{\alpha
,n})_{\alpha}\ntriangleleft Q_{\frac{\mu\alpha}{\sigma}}^n$. The
dichotomy for sequences of Gaussian measures of \cite[Corollary 4]{Eag81}
says that two sequences of Gaussian measures on $\mathbb{R}^n$ are
either mutually contiguous or entirely separable.
So we conclude that, for each $n> 1$, $(Q^{\alpha,n})_{\alpha
>0}\vartriangle(Q_{\frac{\mu\alpha}{\sigma}}^n)_{\alpha>0}$.
\end{proof}
%
\begin{remark}\label{final}
From Remark \ref{rmk0}, when $\mu=0$, the same arguments lead to
the conclusion that Lemma \eqref{finiten} holds true for $n=1$.
\end{remark}
\begin{proof}[Proof of Theorem~\ref{t1}]
From Proposition \ref{pkk} (see also \cite[Proposition 4]{Kakra}), we
know that there is a SAA if and only if $(Q^{\alpha})_{\alpha
>0}\vartriangle(Q_{\frac{\mu\alpha}{\sigma}})_{\alpha>0}$.

Fix $n>1$. By Lemma \ref{finiten}, there exist a sequence $(\alpha
_\ell)_{\ell\geq1}$ converging to infinity and sets $A_\ell\in\Fs
_n$ such that
\[
\lim_{\ell\to\infty} Q^{\alpha_\ell}(A_\ell)=\lim
_{\ell\to
\infty}Q^{\alpha_\ell,n}(A_\ell)=0
\]
and
\[
\lim_{\ell\to\infty} Q_{\frac{\mu\alpha_\ell}{\sigma}}(A_\ell )=\lim
_{\ell\to\infty}Q_{\frac{\mu\alpha_\ell}{\sigma
}}^n(A_\ell)=1.
\]
The result follows.
\end{proof}

\section{Interpretation of the results in the restricted
markets}\label{S3}
Lemma~\ref{finiten} might suggest that, for each $n>1$ (or following
Remark \ref{final}, for each $n\geq1$ if $\mu=0$), there exists also
some kind of asymptotic arbitrage in the large financial market
consisting of the restrictions of the $\alpha$-mixed fractional
Black--Scholes models, $\alpha>0$, to the grid $E_n:=\{0,\frac
{1}{n},\ldots,\frac{n-1}{n},1\}$. However, we will show that this is impossible.

For simplicity, we only consider the case $n=1$ and $\mu=0$.
We also assume that $S^{H,\alpha}_0=1$. Thus, for each $\alpha>0$,
the corresponding market is
%
\begin{equation}
S^{H,\alpha}_t=\exp{ \biggl(\sigma \biggl(Z^H_t+
\frac{1}{\alpha} B_t \biggr)-\frac{\sigma^2}{2\alpha^2}t \biggr)},\quad
t=0,1.\label{2step}
\end{equation}
In this case all possible strategies are constants and hence the value
process $V^{\alpha}_1$ takes the following form
\[
V^{\alpha}_1=c_{\alpha}\bigl(S^{H,\alpha}_1-1
\bigr),
\]
where $c_{\alpha}\in\mathbb{R}$. Obviously, we cannot hope for
admissibility (boundedness from below), see the discussion in the
introduction of \cite{Mi}. But even if we do not require any
admissibility here, there is no way to choose a sequence of $\alpha
_\ell\to\infty$ and corresponding value processes $V^{\alpha_\ell
}$ such that the following hold: there exists $\beta>0$ and $\ep
_\ell\to0$ with
%
\begin{align}
\label{nafl} (i)\quad & P\bigl(V^{\alpha_\ell}_1>\beta\bigr)>
\beta, \text{ for all $\ell $},
\nonumber
\\
(ii)\quad & \lim_{\ell\to\infty}P\bigl(V^{\alpha_\ell}_1\geq-
\ep_\ell\bigr)=1.
\end{align}
This is not possible since $Z^{H}_1$ as well as $B_1$ are independent
$N(0,1)$ and hence are strictly positive as well as strictly negative,
with positive $P$-probability (and here neither letting $\alpha\to
\infty$ nor multiplying $S^{H,\alpha}_1-1$ by some constants, either
positive or negative, will be of any help: whenever there will be a
strictly positive part in the limit there will also be a strictly
negative part in the limit with a non-disappearing probability). Hence
there is no such thing as (\ref{nafl}) which, in our discrete time
$t=0,1$ situation, is the appropriate version
of an asymptotic arbitrage.

The reason behind this apparent contradiction is that in contrast to
the continuous time large financial market its discrete counterpart is
not complete.
Under the original measure $P$ (which induces $Q^{\alpha}$ on
$\mathcal{C}[0,1]$) we have that $Z^H_1\sim N(0,1)$ and $B_1\sim
N(0,1)$ and the two random variables are independent. We know that the
Wiener measure $Q_W$ is a martingale measure for $S^{\alpha}$
(understood on $\mathcal{C}[0,1]$) for all $\alpha$, hence
$Q_W|_{\mathcal{F}_1}$ is an equivalent martingale measure for (\ref
{2step}). We will now construct a different martingale measure for the
process (\ref{2step}) which is equivalent to $P$ on $(\varOmega,\mathcal{F})$.

Indeed, define a measure $\tilde{P}$ on $(\varOmega,\mathcal{F})$ as
follows: $\frac{d\tilde{P}}{dP}=g(X)$ where we have $X:=\exp(\sigma
Z^H_1)$ and $g(x)=e^{-x}\frac{1}{h(x)}$ where $h(x)=\frac{1}{\sqrt{2\pi
}\sigma x} \exp(-\frac{1}2(\frac{\ln(x)}{\sigma})^2)$ is the density
of a lognormal distribution, i.e., the density of the law of $X$ under
$P$. Obviously this measure change has the purpose to make the
distribution of $X$ exponential with parameter 1. Recall that the
measure $Q^{\alpha,1}$ is considered as a measure on $\Rb$ (the
measure induced by $M_1^{H,\alpha}:=\alpha Z_1^H+B_1$).

\begin{lemma}\label{newemm}
The measure $\tilde{P}$ satisfies:
\begin{enumerate}
\item\xch{$\tilde{P}\sim P$.}{$\tilde{P}\sim P$}
\item$E_{\tilde{P}}[S^{H,\alpha}_1]=1=S^{H,\alpha}_0$, which means
that $\tilde{P}$ is a martingale measure for (\ref{2step}), for each
$\alpha$.
\item Let $\tilde{Q}^{\alpha,1}$ be the measure that is induced by
$(M_1^{H,\alpha},\tilde{P})$ on $\Rb$, for each $\alpha>0$. Then
$(\tilde{Q}^{\alpha,1})_{\alpha>0}\triangleleft\triangleright
(Q^{\alpha,1})_{\alpha>0}$.
\end{enumerate}
\end{lemma}

\begin{proof}
To prove (1) observe that $g(x)>0$ for all $0<x<\infty$ and
$0<X<\infty$ $P$-a.s. and hence $\frac{d\tilde{P}^{\alpha
}}{dP}=g(X)>0$ $P$-a.s. Moreover
\[
E_P\bigl[g(X)\bigr]=\int_0^{\infty}g(x)h(x)dx=
\int_0^{\infty}e^{-x}dx=1,
\]
and so $\tilde{P}$ is a probability equivalent to $P$.

For (2) we see that
\begin{align}
E_{\tilde{P}}\bigl[S^{H,\alpha}_1\bigr] &=E_P
\biggl[g \bigl(e^{\sigma
Z^{H}_1} \bigr)\exp{ \biggl(\sigma \biggl(Z^H_1+
\frac{1}{\alpha} B_1 \biggr)-\frac{\sigma^2}{2\alpha^2} \biggr)} \biggr]
\nonumber
\\
&=E_P \bigl[g \bigl(e^{\sigma Z^{H}_1} \bigr)\exp{ \bigl(\sigma
Z^H_1 \bigr)} \bigr]E_P \biggl[\exp{
\biggl(\frac{\sigma}{\alpha} B_1-\frac{\sigma^2}{2\alpha^2} \biggr)} \biggr]
\nonumber
\\
&=E_P\bigl[g(X)X\bigr],
\nonumber
\end{align}
where we used the independence of $Z^{H}_1$ and $B_1$ under $P$.
Finally we have that
\[
E_P\bigl[g(X)X\bigr]=\int_0^{\infty}g(x)xh(x)dx=
\int_0^{\infty}xe^{-x}dx=1,
\]
proving (2).

For (3) note that, for each $A\in\mathcal{B}(\Rb)$, we have
$Q^{\alpha,1}(A)=P(M_1^{H,\alpha}\in A)$ and $\tilde{Q}^{\alpha
,1}(A)=\tilde{P}(M^{H,\alpha}_1\in A)$. Thus, using that $\tilde
{P}\sim P$, we infer, for a family of sets $A^{\alpha}$, that
$Q^{\alpha,1}(A^{\alpha})=P(D^{\alpha})\to0$, for $\alpha\to
\infty$, if and only if $\tilde{Q}^{\alpha,1}(A^{\alpha})=\tilde
{P}(D^{\alpha})\to0$, where $D^{\alpha}=\{M_1^{H,\alpha}\in
A^{\alpha}\}\in\mathcal{F}_1$. The result follows.
\end{proof}

In conclusion, Lemma~\ref{newemm} shows that there exists a family of
equivalent martingale measures for the model (\ref{2step}) with good
properties, in this case with the property of mutual contiguity.\querymark{Q4} And this fact is reflected
by the impossibility to find asymptotic arbitrage opportunities for the
family of models (\ref{2step}), $\alpha>0$.

\begin{appendix}

\section{Relative entropy}\label{A1}
In this section, we recall the concept of relative entropy and some
equivalent characterization. A more detailed presentation of the topic
can be found in \cite{Hihi}.
%
\begin{definition}\label{a1.1}
Let $Q_1$ and $Q_2$ be probability measures on a measurable space
$(\varOmega,\Fs)$ and let $P=\{F_i: i=1,\ldots,n\}$ be a finite
partition of $\varOmega$, i.e. $\varOmega=\cup_{i=1}^nF_i$ and $F_i$ are
pairwise disjoint. The entropy of the measure $Q_1$ relative to $Q_2$
is the quantity
\[
H(Q_1|Q_2)=\sup_{\Ps}\sum
_{j=1}^nQ_1(F_j)\ln
\biggl(\frac
{Q_1(F_j)}{Q_2(F_j)} \biggr),
\]
where $\Ps$ is the class of all possible finite partitions $P$ of
$\varOmega$. In the above formula, we assume that $0\ln0=0$ and $\ln
0=-\infty$.
\end{definition}

\begin{lemma}[{\cite[Lemma 6.1]{Hihi}}]\label{a1.2}
If a probability measure $Q_1$ is absolutely continuous w.r.t. another
probability measure $Q_2$, then the relative entropy $H(Q_1|Q_2)$ is
related to the Radon--Nikodym derivative $\phi=\frac{dQ_1}{dQ_2}$ as follows:
\[
H(Q_1|Q_2)=E_{Q_1}\bigl[\ln(\phi)
\bigr]=E_{Q_2}\bigl[\phi\ln(\phi)\bigr].
\]
\end{lemma}
\end{appendix}

\begin{acknowledgement}
We thank an anonymous referee whose comments and suggestions contributed to
the quality of this version of the paper.
\end{acknowledgement}

\begin{funding}
The third author gratefully
acknowledges financial support from the
\gsponsor[id=GS1,sponsor-id=501100002428]{Austrian Science
Fund} (FWF): \gnumber[refid=GS1]{J3453-N25}.
\end{funding}

\end{document}